\documentclass[12 pt, reqno]{amsart}
\usepackage{amsmath,times,epsfig,amssymb,amsbsy,amscd,amsfonts,amstext,graphicx,color,bm,amsthm}
\usepackage[all]{xy}
\usepackage{graphicx}
\usepackage{hyperref}

\usepackage{tikz,subfigure}
\usepackage{tikz-cd}

\newcommand{\KK}{\mathbb{K}}
\newcommand{\NN}{\mathbb{N}}
\newcommand{\RR}{\mathbb{R}}

\newcommand{\A}{\mathcal{A}}

\newcommand \ideal[1] {\langle #1 \rangle}

\newtheorem{Theorem}{Theorem}[section]
\newtheorem{Definition}[Theorem]{Definition}

\newtheorem{Lemma}[Theorem]{Lemma}
\newtheorem{Proposition}[Theorem]{Proposition}
\newtheorem{Corollary}[Theorem]{Corollary}
\newtheorem{Remark}[Theorem]{Remark}
\newtheorem{Example}[Theorem]{Example}
\newtheorem{Conjecture}[Theorem]{Conjecture}

\DeclareMathOperator{\rk}{rk}
\DeclareMathOperator{\Der}{Der}

\DeclareMathOperator{\pdeg}{pdeg}

\DeclareMathOperator{\gin}{gin}
\DeclareMathOperator{\rgin}{rgin}
\DeclareMathOperator{\LT}{LT}
\DeclareMathOperator{\GL}{GL}
\DeclareMathOperator{\HF}{HF}
\DeclareMathOperator{\DegRevLex}{DegRevLex}

\DeclareMathOperator{\reg}{reg}

\begin{document}

\title{Lefschetz Properties and Hyperplane Arrangements}

\begin{abstract} 
In this article, we study the weak and strong Lefschetz properties, and the related notion of almost revlex ideal, in the non-Artinian case, proving that several results known in the Artinian case hold also in this more general setting. We then apply the obtained results to the study of the Jacobian algebra of hyperplane arrangements.
\end{abstract}

\author{Elisa Palezzato}
\address{Elisa Palezzato, Department of Mathematics, Hokkaido University, Kita 10, Nishi 8, Kita-Ku, Sapporo 060-0810, Japan.}
\email{palezzato@math.sci.hokudai.ac.jp}
\author{Michele Torielli}
\address{Michele Torielli, Department of Mathematics, GI-CoRE GSB, Hokkaido University, Kita 10, Nishi 8, Kita-Ku, Sapporo 060-0810, Japan.}
\email{torielli@math.sci.hokudai.ac.jp}


\date{\today}
\maketitle



\section{Introduction}

The weak and strong Lefschetz properties are connected to many topics in algebraic geometry, commutative algebra and combinatorics. Some of these connections are quite surprising and there are still several open questions. 
In addition, most of the research in this area focus on the Artinian case. We refer to \cite{harima2013lefschetz} for an overview of the Lefschetz properties in the Artinian case, and to \cite{migliore2013survey} for several open questions in the area.

The study of Lefschetz properties has already been linked to the theory of free hyperplane arrangements and to Terao's conjecture. See for example \cite{maeno2011strong}, \cite{numata2007strong}, \cite{di2014singular} and \cite{ilardi2018jacobian}.

The goal of this paper is to start the study of the weak and strong Lefschetz properties, and the related notion of almost revlex ideal, in the non-Artinian case, and then apply the obtained results to the study of the Jacobian algebra of a hyperplane arrangement.

This paper is organized as follows. In Section 2, we describe the notions of weak and strong Lefschetz properties and their basic attributes. In Section 3, we illustrate the connection between the notion of $x_l$-chains and the strong Lefschetz property. In Section 4, we connect the non-Artinian case to the Artinian one. In Section 5, we recall the notion of almost revlex ideal and we put it in connection with the Lefschetz properties. In Section 6, we describe how having the weak Lefschetz property gives information on the graded Betti numbers. In Section 7, we recall the definitions and basic properties of hyperplane arrangements. In Section 8, we analyze the Jacobian algebra of an arrangement from the Lefschetz properties point of view.


\section{Lefschetz properties}


Let $\KK$ be a field of characteristic $0$ and consider the polynomial ring $S=\KK[x_1,\dots,x_l]$. 

\begin{Definition} 
A monomial ideal $I$ of $S$ is said to be  \textbf{strongly stable} if for every power-product $t \in I$ and
every $i,j$ such that $1\le i<j\le l$ and $x_j|t$, the power-product $x_i\cdot t/x_j\in I$.
\end{Definition}

\begin{Example}\label{ex:nonss-ssideal}
The ideal $I=\ideal{x^3, x^2y, xy^2, xyz}$ is not a strongly stable in $\mathbb{Q}[x,y,z]$ because $x\cdot xyz/y = x^2z\not\in I$. It is enough to add $x^2z$ as a minimal generator to $I$ to obtain a strongly stable ideal.
\end{Example}

\begin{Definition} 
Let $\sigma$ be a term ordering on $S$ and $f$ a non-zero polynomial in $S$. 
Then $\LT_\sigma(f)= \max_\sigma\{{\rm Supp}(f)\}$, where ${\rm Supp}(f)$ is the set of all power-products 
appearing with non-zero coefficient in $f$.  If
$I$ is an ideal in $S$, then the 
\textbf{leading term ideal} (or \textit{initial ideal})
of $I$ is the ideal $\LT_\sigma(I)$ of $S$ generated by
 $\{\LT_\sigma(f) \mid f \in I{\setminus}\{0\}\;\}$.
\end{Definition}

The following theorem is due to Galligo \cite{galligo1974propos}.

\begin{Theorem}[\cite{galligo1974propos}]\label{thm:Galligo}
Let $I$ be a homogeneous ideal of $S$, with $\sigma$ a term ordering such that
$x_1>_\sigma x_2 >_\sigma \dots >_\sigma x_l$. Then there exists a Zariski
open set $U\subseteq\GL(l)$ and a strongly stable ideal $J$ such that for each  $g\in U$, $\LT_\sigma(g(I)) = J$.
\end{Theorem}

\begin{Definition} 
The strongly stable ideal $J$ given in Theorem~\ref{thm:Galligo} is called the \textbf{generic initial ideal} with respect to $\sigma$ of $I$ and it is denoted by {\boldmath$\gin_\sigma(I)$}.  
In particular, when $\sigma= \DegRevLex$, $\gin_\sigma(I)$ is simply denoted by {\boldmath$\rgin(I)$}.
\end{Definition}

As described in \cite{BPT2016}, we can read a lot of information on an ideal from its generic initial ideal. For example, we have the following.
\begin{Remark}\label{rem:rginsameHFasideal} 
Let $I$ be a homogeneous ideal of $S$. Then the Hilbert function of $S/I$ coincides with the one of $S/\rgin(I)$.
\end{Remark}

\begin{Definition}\label{def:WLPandSLP} 
Let $R$ be a graded ring over $\KK$, and $R = \bigoplus_{i\geq 0} R_i$ its decomposition into homogeneous
components with $\dim_\KK (R_i) < \infty$.
\begin{enumerate}
\item The graded ring $R$ is said to have the \textbf{weak Lefschetz property (WLP)}, 
if there exists an element $\ell \in R_1$ such that
the multiplication map
\begin{align*}
\times \ell \colon R_i &\rightarrow R_{i+1}\\ 
f &\mapsto \ell f
\end{align*}
 is full-rank for every $i \geq 0$. In this case, $\ell$ is called a \textbf{weak Lefschetz element}.
\item The graded ring $R$ is said to have the \textbf{strong Lefschetz property (SLP)}, 
if there exists an element $\ell \in R_1$ such that
the multiplication map
\begin{align*}
\times \ell^s \colon R_i &\rightarrow R_{i+s}\\ 
f &\mapsto \ell^sf
\end{align*}
is full-rank for every $i \geq 0$ and $s\geq 1$. In this case, $\ell$ is called a \textbf{strong Lefschetz element}.
\end{enumerate}
\end{Definition}


Similarly to Lemma 2.7 of \cite{wiebe2004lefschetz}, for strongly stable ideals is easier to check if the associated quotient algebra has the SLP or the WLP.

\begin{Lemma}\label{lemma:SSWLPx_l}
Let $I$ be a strongly stable ideal of $S$.
Then $S/I$ has the SLP (respectively the WLP) if and only if $~S/I$ has the SLP (respectively the WLP) with Lefschetz element $x_l$. 
\end{Lemma}

\begin{proof}
For any element $\ell \in S_1$, $j \in \NN$ and $k\geq 1$ define
$$c_\ell(j,k) = \dim_{\KK}((I :_{S} \ell^k)_{j})$$ 
and
$$\alpha(j,k) = \max\{\dim_{\KK}(I_j),\dim_{\KK}(S_j)-\dim_{\KK}((S/I)_{j+k})\},$$
where $(I :_{S} \ell^k)=\{f\in S~|~\ell^kf\in I\}$.
For any $\ell \in S_1$ we have that $c_\ell(j,k) \geq \alpha(j,k)$.
It is clear that $\ell$ is a weak (respectively strong) Lefschetz element
on $S/I$ if and only if $c_\ell(j,1) = \alpha(j,1)$ for all $j \in \NN$ (respectively $c_\ell(j,k) = \alpha(j,k)$ for all $j \in \NN$ and all $k \geq 1$). 
Since $I$ is strongly stable, then $(I :_S x^k_l) \subseteq (I :_S \ell^k)$ for every $\ell \in S_1$ and all $k \geq 1$. 
This implies $c_{x_l}(j,k) \leq c_\ell(j,k)$ for all $j \in \NN$ and all $k \geq 1$, and hence we obtain the equivalence.
\end{proof}

Similarly to Proposition 2.8 of \cite{wiebe2004lefschetz}, to check if a quotient algebra has the SLP or the WLP, it is enough to check the quotient by strongly stable ideals.
\begin{Proposition}\label{prop:ILPginLP}
Let $I$ be a homogeneous ideal of $S$.
Then the graded ring $S/I$ has the SLP (respectively the WLP) if and only if $S/\rgin(I)$ 
has the SLP (respectively the WLP).
\end{Proposition}

\begin{proof}
For any element $\ell \in S_1$, $j \in \NN$ and $k \geq 1$ define 
$$d_\ell(j,k) = \dim_{\KK}((S/(I,\ell^k))_{j+k})$$
and 
$$\gamma(j,k) = \max\{0, \dim_{\KK}((S/I)_{j+k}) - \dim_{\KK}((S/I)_j)\}.$$
For any $\ell \in S_1$ we have that $d_\ell(j,k) \geq \gamma(j,k)$.
It is clear that $\ell$ is a weak (respectively strong) Lefschetz element on
$S/I$ if and only if $d_\ell(j,1) = \gamma(j,1)$ for all $j \in \NN$ (respectively $d_\ell(j,k) = \gamma(j,k)$ for all $j \in \NN$ and all $k \geq 1$).
In Lemma~1.2 of \cite{conca2003reduction} Conca showed that the Hilbert function of $S/(\rgin(I),x_l)$ is equal to the
Hilbert function of $S/(I,\ell)$ for a general linear form $\ell \in S_1$. 
Together with Theorem 2.4 and Lemma 2.7, this yields the assertion about the weak Lefschetz property.
By slightly generalizing the arguments of Conca's proof (one has to use the fact that $\LT_{\DegRevLex}(gI+(x^k_l))= \LT_{\DegRevLex}(gI)+(x^k_l)$ for all $k \geq 1$), one obtains that the Hilbert function of $S/(\rgin(I),x^k_l)$ is equal to the Hilbert function of $S/(I,\ell^k)$ for a general linear form $\ell \in S_1$ and all $k \geq 1$. This yields the assertion about the strong Lefschetz property.
\end{proof}

Directly from Definition~\ref{def:WLPandSLP} of Lefschetz properties, we have the next two propositions (c.f. \cite[Proposition 3.2, Proposition 3.9]{harima2013lefschetz}).

\begin{Proposition} Let $R$ be a ring over $\KK$ with standard grading. If $R$ has the WLP, then $R$ has an increasing or unimodal Hilbert function.
\end{Proposition}
\begin{proof} Assume $R$ does not have an increasing Hilbert function. Let $\mathfrak{m}$ be the irrelevant maximal ideal of $R$. Since $R$ has standard grading, $\mathfrak{m}^i$ is generated by $R_i$.
Let $k\ge0$ be the smallest integer such that $\dim_\KK(R_k)>\dim_\KK(R_{k+1})$. Since $R$ has the WLP, there exists $\ell \in R_1$ such that the map $\times \ell \colon R_k \rightarrow R_{k+1}$ is surjective, and hence that $\mathfrak{m}^{k+1}=\ell\mathfrak{m}^k$. Moreover, this implies that $\mathfrak{m}^{i+1}=\ell\mathfrak{m}^i$ for all $i\ge k$. Hence, the map
$\times \ell \colon R_i \rightarrow R_{i+1}$ is surjective for all $i\ge k$. This implies that $\dim_\KK(R_i)\ge\dim_\KK(R_{i+1})$ for all $i\ge k$, and hence, the Hilbert function of $R$ is unimodal.
\end{proof}

By \cite[Remark 3.3]{harima2013lefschetz}, the previous Proposition is false if we do not assume that $R$ has standard grading. Moreover,
not all rings with unimodal Hilbert function have the WLP.

\begin{Example} Let $I$ be the strongly stable ideal of $S=\KK[x,y,z,w]$ generated by $\{x^2, xy, xz, xw\}$. The quotient $S/I$ has an increasing Hilbert function. On the other hand, since the map $\times w \colon (S/I)_1 \rightarrow (S/I)_2$ is not injective, $w$ is not a Lefschetz element for $S/I$. By Lemma~\ref{lemma:SSWLPx_l}, $S/I$ does not have the WLP.
\end{Example}

\begin{Proposition} Let $R$ be a ring over $\KK$ with any grading. If $R$ has the SLP, then $R$ has an increasing or unimodal Hilbert function.
\end{Proposition}
\begin{proof} Assume that $R$ does not have an unimodal or increasing Hilbert function. Hence, assume that there exist three integers $k<l<m$
such that $\dim_\KK(R_k)>\dim_\KK(R_l)<\dim_\KK(R_m)$.
This implies that the map $\times \ell^{m-k} \colon R_k \rightarrow R_m$ cannot have full rank for any element $\ell \in R_1$. Hence, $R$ cannot have the SLP.
\end{proof}

\section{$x_l$-chains and the strong Lefschetz property}

In this section, we will consider the polynomial ring $S=\KK[x_1,\dots,x_l]$  with $\sigma= \DegRevLex$ as term ordering.

\begin{Definition} 
Let $I$ be a homogeneous ideal of $S$. A \textbf{$x_l$-chain} with respect to $I$ is a sequence
$$\{tx_l^k~|~k\ge0, tx_l^k\notin\LT(I)\},$$
where $t$ is a power-product in $S$ not divisible by $x_l$.
\end{Definition}

In general, we have an infinite number of distinct $x_l$-chains and an $x_l$-chain might be infinite. If we assume that $S/I$ is Artinian, then 
each $x_l$-chain is finite and there is only a finite number of them.

Here we give a necessary and sufficient condition for a quotient ring $S/I$ to have the SLP in terms of $x_l$-chains.


\begin{Definition} 
Let $I$ be a homogeneous ideal of $S$. We say that $I$ has the \textbf{strong Lefschetz condition} (\textbf{SL condition}) if
for any two power-products $u,v\in S\setminus\LT(I)$ not divisible by $x_l$ with $\deg(u)<\deg(v)$, if the $x_l$-chain $v, vx_l, vx_l^2, \dots, vx_l^r$ is finite, then the $x_l$-chain $u, ux_l, ux_l^2, \dots, ux_l^s$ is finite and $\deg(ux_l^s)\ge\deg(vx_l^r)$.
\end{Definition}

In \cite[Lemma 7]{harima2009generic}, the authors proved a similar statement for the Artinian case.
\begin{Lemma} 
Let $I$ be a homogeneous ideal of $S$. The following three conditions are equivalent
\begin{enumerate}
\item $S/I$ has the SLP with strong Lefschetz element $x_l$.
\item $S/\LT(I)$ has the SLP with strong Lefschetz element $x_l$.
\item $I$ has the SL condition.
\end{enumerate}
\end{Lemma}
\begin{proof} Since the SL condition is exactly the condition for $x_l$ to be a strong Lefschetz element for $S/\LT(I)$, we have that conditions (2) and (3) are clearly equivalent. We will now prove that conditions (1) and (2) are equivalent.

We have that $\LT(I:x_l^s)=\LT(I):x_l^s$, for all $s\ge1$. On the other hand, we have the following two isomorphisms
$$\ker(\times x_l^s\colon S/I\to S/I)=I:x_l^s/I,$$
$$\ker(\times x_l^s\colon S/\LT(I)\to S/\LT(I))=\LT(I):x_l^s/\LT(I),$$
and hence, we have that
$$\rk(\times x_l^s\colon (S/I)_i\to (S/I)_{i+s})= \rk(\times x_l^s\colon (S/\LT(I))_i\to (S/\LT(I))_{i+s}),$$
for all $i\ge0$ and $s\ge 1$.
Therefore, it follows from the formula above that (1) and (2) are equivalent.
\end{proof}

\section{Regularity and weak Lefschetz property}

\begin{Definition} 
Let $I$ be a homogeneous ideal of $S$. 
The \textbf{Castelnuovo-Mumford regularity} of $I$, denoted~{\boldmath$\reg(I)$},
is the maximum of the numbers $d_i-i$,
where  $d_i = \max\{j\mid \beta_{i,j}(I)\ne0\}$ and $\beta_{i,j}(I)$ are the graded  
Betti numbers of $I$.
\end{Definition}
In \cite{bayer1987criterion}, the authors described the connection between the Castelnuovo-Mumford regularity of an ideal and the maximal degree of the minimal generators of its generic initial ideal.
\begin{Theorem}[\cite{bayer1987criterion}]\label{theo:regginequalideal} 
Let $I$ be a homogeneous ideal of $S$. Then $\reg(I)=\reg(\rgin(I))$. 
Moreover, if $I$ is a strongly stable ideal, then $\reg(I)$ is the highest degree of a minimal generator of $I$.
\end{Theorem}

In principle, in order to understand if a ring has the WLP, one has to check an infinite number of multiplication maps. However, the following result shows that it is enough to check a finite number of them.
\begin{Theorem}\label{theo:wlplessregI} 
Let $I$ be a homogeneous ideal of $S$. Then the following facts are equivalent
\begin{enumerate}
\item the graded ring $S/I$ has the WLP.
\item There exists an element $\ell \in (S/I)_1$ such that the multiplication map $\times \ell \colon (S/I)_i \rightarrow (S/I)_{i+1}$ has full-rank for every $0\le i \le \reg(I)-1$.
\end{enumerate}
\end{Theorem}
\begin{proof} By Proposition~\ref{prop:ILPginLP} and Theorem~\ref{theo:regginequalideal}, it is enough to prove the statement when $I$ is a strongly stable ideal.
By Definition~\ref{def:WLPandSLP} of WLP, (1) clearly implies (2). On the other hand, if condition (2) is satisfied, we only need to show that the multiplication map $\times \ell \colon (S/I)_i \rightarrow (S/I)_{i+1}$ has full-rank for every $i \ge \reg(I)$. However, this is a consequence of the fact that if $i> \reg(I)$, then $I$ has no minimal generators in degree $i$ and the fact that the Hilbert function $\HF(S/I, i)$ is increasing for all $i \ge \reg(I)$.
\end{proof}

The following result shows that if we are interested in the WLP, we can always reconduct to the Artinian case.   

\begin{Corollary} 
Let $I$ be a homogeneous ideal of $S$. Then the following facts are equivalent
\begin{enumerate}
\item the graded ring $S/I$ has the WLP with Lefschetz element $\ell$.
\item the graded Artinian ring $S/J$ has the WLP with Lefschetz element $\ell$, where $J=I+(x_1,\dots,x_l)^{\reg(I)+1}$.
\end{enumerate}
\end{Corollary}
\begin{proof} By Theorem~\ref{theo:regginequalideal}, $\rgin(I)$ has no minimal generators of degree greater or equal to $\reg(I)+1$. Hence $\rgin(J)=\rgin(I)+(x_1,\dots,x_l)^{\reg(I)+1}$.

Assume now that $S/I$ has the WLP, then by Proposition~\ref{prop:ILPginLP} and the equality $\rgin(J)=\rgin(I)+(x_1,\dots,x_l)^{\reg(I)+1}$,
we have that the multiplication map $\times \ell \colon (S/J)_i \rightarrow (S/J)_{i+1}$ has full-rank for every $0\le i \le \reg(I)-1$. Since $(S/J)_i=0$ for all $i\ge\reg(I)+1$, we only need to show that the multiplication map $\times \ell \colon (S/J)_{\reg(I)} \rightarrow (S/J)_{\reg(I)+1}$ has full-rank. However, this map is obviously surjective since $(S/J)_{\reg(I)+1}=0$.

Assume now that $S/J$ has the WLP. Clearly, $(S/\rgin(I))_i=(S/\rgin(J))_i$ for all $0\le i \le \reg(I)$. Then by Proposition~\ref{prop:ILPginLP} and Theorem~\ref{theo:wlplessregI}, $S/I$ has the WLP with Lefschetz element $\ell$.
\end{proof}

\section{Almost revlex ideals and Lefschetz properties}

If we consider strongly stable ideals, then studying $x_l$-chains gives us information on the minimal generators of the ideal.
\begin{Lemma}\label{lemma:ssiandgenstandard} 
Let $I$ be a strongly stable ideal of $S$ and $u, ux_l, \dots, ux_l^{s-1}$ a finite $x_l$-chain with respect to $I$. Then
$ux_l^{s}$ is a minimal generator of $I$.
\end{Lemma}
\begin{proof}
There exists a minimal generator $vx_l^r$ such that $v$ divides $u$ and $r\le s$. If $r < s$, then $vx_l^r|ux_l^{s-1}$, and this contradicts the condition that $ux_l^{s-1}\notin\LT(I)=I$. Therefore, we have $r = s$. Assume that $\deg(v) < \deg(u)$. Then there exists $i < l$ such that $x_iv|u$, and hence $x_ivx_l^{s-1}|ux_l^{s-1}$. By the definition of strongly stable ideal, $vx_l^s\in I$ implies that $x_ivx_l^{s-1}\in I$,
and we have $ux_l^{s-1}\in I$. This is a contradiction, and we have $v = u$. This means that $ux_l^s$ is a minimal generator of I.
\end{proof}

\begin{Definition}\label{def:ARL} 
A monomial ideal $I$ of $S$ is called an \textbf{almost revlex ideal}, if for any power-product $t$ in the minimal generating set of $I$, every other power-product $t'$ of $S$ with $\deg(t')=\deg(t)$ and $t'>_{\DegRevLex}t$ belongs to $I$.
\end{Definition}

\begin{Remark} 
Every almost revlex ideal is strongly stable.
\end{Remark}

\begin{Example}
Consider the ideal $I=\ideal{x^3, x^2y, xy^2, xyz}$ in $\mathbb{Q}[x,y,z]$ of Example~\ref{ex:nonss-ssideal}. As seen before, it is not a strongly stable, and hence it is not almost revlex. On the other hand,  also the strongly stable ideal $J=I+\ideal{x^2z}$ is not almost revlex.
In fact, $xyz\in J$, but $y^3\notin J$. If we consider the ideal $J+\ideal{y^3}$, finally, this is an almost revlex ideal. 
\end{Example}

\begin{Remark}\label{rem:ARLsameHFsameideal} 
If two almost revlex ideals have the same Hilbert function, then they coincide.
\end{Remark}

As in \cite[Proposition 12]{harima2009generic} for the Artinian case, to check if an ideal is almost revlex, we can use a sort of inductive method.
\begin{Proposition}\label{prop:ARLxlchaincond} 
Let $I$ be a monomial ideal of $S$ and set $\bar{S}=\KK[x_i,\dots,x_{l-1}]$. Then $I$ is an almost revlex ideal if and only if the following two conditions are satisfied
\begin{enumerate}
\item $I\cap\bar{S}$ is an almost revlex ideal.
\item\label{item:condiiARL} For any two power-products of $S$ $u,v\notin I$ not divisible by $x_l$ with $u<_{\DegRevLex}v$, the ending degree of the $x_l$-chain starting with $u$ is greater or equal to the ending degree of the $x_l$-chain starting with $v$.
\end{enumerate}
\end{Proposition}
\begin{proof} 
Assume first that $I$ is an almost revlex ideal. It is clear that $I\cap\bar{S}$ is an almost revlex ideal. Let $u,v\notin I$ be
two power-products of $S$ not divisible by $x_l$ with $u<_{\DegRevLex}v$. This implies that $\deg(u)\le\deg(v)$. Let $s\ge\deg(v)$ and assume that $vx_l^{s-\deg(v)}\notin I$. In order to prove (2), it is enough to show that $ux_l^{s-\deg(u)}\notin I$. 
Assume by absurd, $ux_l^{s-\deg(u)}\in I$, then by Lemma~\ref{lemma:ssiandgenstandard} there exists an integer $\deg(u)< k\le s$ such that $ux_l^{k-\deg(u)}$ is a minimal generator of $I$. Since $I$ is an almost revlex ideal, if $k\ge\deg(v)$, then $vx_l^{k-\deg(v)}\in I$ and hence $vx_l^{s-\deg(v)}\in I$. However, this is a contradiction, and hence $k<\deg(v)$.
In this case, since $I$ is an almost revlex ideal and $ux_l^{k-\deg(u)}$ is a minimal generator of $I$, all the power-products in $x_1,\dots,x_{l-1}$ of degree $k$ belongs to $I$. This implies that $v\in I$. However, this is a contradiction, and hence, $ux_l^{s-\deg(u)}\notin I$.

Assume now that $I$ satisfy conditions (1) and (2), and let $ux_l^s$ be a minimal generator of $I$, where $u$ is a power-product not divisible by $x_l$ and $s\ge0$. We have to show that every power-product of the same degree as $ux_l^s$ which is greater than $ux_l^s$ with respect to $\DegRevLex$ belongs to $I$. If $s=0$, then the statement is true by (1). Assume $s>0$. This implies that $u\notin I$. Consider
a power-product $vx_l^k$ such that $v$ is not divisible by $x_l$, $\deg(vx_l^k)=\deg(ux_l^s)$ and $vx_l^k>_{\DegRevLex}ux_l^s$. We want to show that $vx_l^k\in I$. Since $vx_l^k>_{\DegRevLex}ux_l^s$, then $v>_{\DegRevLex}u$ if $k=s$ or $\deg(v)>\deg(u)$ if $k<s$. In both case $v>_{\DegRevLex}u$. Then by condition (2), $vx_l^k\in I$.
\end{proof}

Similarly to \cite[Corollary 13]{harima2009generic} for the Artinian case, since condition (2) of Proposition~\ref{prop:ARLxlchaincond} in the case where $\deg(u) < \deg(v)$ is exactly the SL condition, we have the following corollary.

\begin{Corollary}\label{corol:ALMhasSLP} 
Let $I$ be an almost revlex ideal of $S$. Then $S/I$ has the SLP with Lefschetz element $x_l$.
\end{Corollary}

In \cite[Theorem 14]{harima2009generic}, the authors studied the opposite questions for Artinian ideals in $\KK[x,y,z]$.
\begin{Theorem}\label{Theo:SSin3dimSLP} 
Let $I$ be a strongly stable ideal of $S=\KK[x,y,z]$ such that $S/I$ has the SLP. Then $I$ is an almost revlex ideal and it is uniquely determined by the Hilbert function.
\end{Theorem}
\begin{proof} 
To prove the statement, we will check conditions (1) and (2) of Proposition~\ref{prop:ARLxlchaincond}.
Since strongly stable ideals are almost revlex ideals in $\KK[x,y]$, then condition (1) of Proposition~\ref{prop:ARLxlchaincond} holds.

Let $u,v\notin I$ be two power-products of $\KK[x,y,z]$ not divisible by $z$ with $u<_{\DegRevLex}v$. We want to show that the ending degree of the $z$-chain starting with $u$ is greater or equal than the ending degree of the $z$-chain starting with $v$.
If $\deg(u)<\deg(v)$, the condition (2) of Proposition~\ref{prop:ARLxlchaincond} is exactly the SL condition. By Lemma~\ref{lemma:SSWLPx_l}, $S/I$ has the SLP with strong Lefschetz element $z$, and hence the SL condition holds.
Assume that $\deg(u)=\deg(v)=k$, and write $u=x^ay^{k-a}$, $v=x^by^{k-b}$, for same $a<b$. Since $I$ is strongly stable, if $vz^s\notin I$, then also $uz^s\notin I$. This implies that the condition (2) of Proposition~\ref{prop:ARLxlchaincond} holds.

Finally, $I$ is determined only by the Hilbert function, by Remark~\ref{rem:ARLsameHFsameideal}.
\end{proof}

Notice that Theorem~\ref{Theo:SSin3dimSLP} is false if we only assume  that $S/I$ has the WLP and not the SLP, or if $\dim(S)\ge4$.

\begin{Example} 
Consider $I$ the ideal of $\KK[x,y,z]$ generated by
$$\{x^5, x^4y, x^3y^2, x^2y^3, xy^4, y^5, x^4z, x^3yz, x^2y^2z, x^3z^2, x^2yz^2\}.$$
Then $I$ is a strongly stable ideal such that $S/I$ has the WLP but not the SLP. However, it is not almost revlex,
in fact $x^2yz^2\in I$ but $xy^3z\notin I$.
\end{Example}

\begin{Example}\label{ex:slpbutnotalmostrevlex} 
Consider $I=\ideal{x^2,xy,xz}$ as ideal of $\KK[x,y,z,w]$.
Then $I$ is a strongly stable ideal such that $S/I$ has the SLP with Lefschetz element $w$. However, $I$ is not an almost revlex ideal since $xz\in I$ and $y^2\notin I$.
\end{Example}

Putting together Proposition~\ref{prop:ILPginLP} and Theorem~\ref{Theo:SSin3dimSLP}, we obtain the following result.

\begin{Corollary} 
Let $I$ be a homogeneous ideal of $S=\KK[x,y,z]$ such that $S/I$ has the SLP. Then $\rgin(I)$ is an almost revlex ideal
and it is uniquely determined by the Hilbert function of $I$.
\end{Corollary}

\section{Graded Betti numbers and weak Lefschetz property}

Before stating the main result of this section, we recall the following result from \cite{eliahou1990minimal}, as described in Corollary 7.2.3 of \cite{herzog2011monomial}.

\begin{Proposition}\label{prop:elihkercrit} Let $I$ be a strongly stable ideal of $S$. Then 
$$\beta_{i,i+j}(S/I)=\sum_{k=1}^l\binom{k-1}{i-1} m_{k,j},$$
where $m_{k,j}$ is the number of minimal generators of $I$ of degree $j$ such that the biggest variable that divides them is $x_k$.
\end{Proposition} 

In order to state the main result, we need a notation for the first differences of the Hilbert function $\HF$.
\begin{Definition} Let $I$ be a homogeneous ideal of $S$. Then, for all $j\ge0$, define
$$c_j(S/I)=\max\{\HF(S/I,j-1)-\HF(S/I,j), 0\},$$
where we put $\HF(S/I,-1)=0$.
\end{Definition}

Similarly to \cite[Proposition 29]{harima2009generic}, if we have a strongly stable ideal $I$ whose quotient algebra $S/I$ has the WLP, then we can compute the graded Betti numbers $\beta_{i,i+j}(S/I)$ in a sort of inductive way.

\begin{Theorem} Let $I$ be a strongly stable ideal of $S$, $\bar{S}=\KK[x_i,\dots,x_{l-1}]$ and $\bar{I}=\bar{S}\cap I$. Assume that $S/I$ has the WLP, then 
$$\beta_{i,i+j}(S/I)=\beta_{i,i+j}(\bar{S}/\bar{I})+\binom{l-1}{i-1}c_j(S/I).$$
\end{Theorem}
\begin{proof} By Lemma~\ref{lemma:SSWLPx_l}, we can assume that $S/I$ has the WLP with Lefschetz element $x_l$.
By Proposition~\ref{prop:elihkercrit}, we have that
$$\beta_{i,i+j}(S/I)=\beta_{i,i+j}(\bar{S}/\bar{I})+\binom{l-1}{i-1} m_{l,j}.$$

Assume first that $I$ has no minimal generators divisible by $x_l$, and hence, $m_{l,j}=0$, for all $j\ge0$. This implies that $\beta_{i,i+j}(S/I)=\beta_{i,i+j}(\bar{S}/\bar{I})$. Moreover, we have that the multiplication by $x_l$ is always injective, and hence, $S/I$ has an increasing Hilbert function. This implies that $c_j(S/I)=0$ for all $j\ge0$, and hence the theorem holds.

Assume now that $I$ has some minimal generators divisible by $x_l$. Notice that if $I$ has some minimal generators divisible by $x_l$ of degree $j$, then the multiplication map by $x_l$ from $(S/I)_{j-1}$ to $(S/I)_j$ cannot be injective. Since $S/I$ has the WLP, we have that 
the multiplication map by $x_l$ from $(S/I)_{j-1}$ to $(S/I)_j$ is surjective. This implies that $c_j(S/I)=\dim_\KK((S/I)_{j-1})-\dim_\KK((S/I)_j)$ coincides with the dimension of the kernel of the multiplication map by $x_l$, and hence, $m_{l,j}=c_j(S/I)$. This shows that also in this situation the theorem holds.
\end{proof}


\section{Preliminaries on hyperplane arrangements}\label{sec:arr}

In this section, we recall the terminology, the basic notations and some fundamental results related to hyperplane arrangements.

A finite set of affine hyperplanes $\A =\{H_1, \dots, H_n\}$ in $\KK^l$ 
is called a \textbf{hyperplane arrangement}. For each hyperplane $H_i$ we fix a defining linear polynomial $\alpha_i\in S$ such that $H_i = \alpha_i^{-1}(0)$, 
and let $Q(\A)=\prod_{i=1}^n\alpha_i$. An arrangement $\A$ is called \textbf{central} if each $H_i$ contains the origin of $\KK^l$. 
In this case, each $\alpha_i\in S$ is a linear homogeneous polynomial, and hence $Q(\A)$ is homogeneous of degree $n$. 


We denote by $\Der_{\KK^l} =\{\sum_{i=1}^l f_i\partial_{x_i}~|~f_i\in S\}$ the $S$-module of \textbf{polynomial vector fields} on $\KK^l$ (or $S$-derivations). 
Let $\delta =  \sum_{i=1}^l f_i\partial_{x_i}\in \Der_{\KK^l}$. Then $\delta$ is  said to be \textbf{homogeneous of polynomial degree} $d$ if $f_1, \dots, f_l$ are homogeneous polynomials of degree~$d$ in $S$. 
In this case, we write $\pdeg(\delta) = d$.

\begin{Definition} 
Let $\A$ be a central arrangement in $\KK^l$. Define the \textbf{module of vector fields logarithmic tangent} to $\A$ (or logarithmic vector fields) by
$$D(\A) = \{\delta\in \Der_{\KK^l}~|~ \delta(\alpha_i) \in \ideal{\alpha_i} S, \forall i\}.$$
\end{Definition}

The module $D(\A)$ is obviously a graded $S$-module and we have that $D(\A)= \{\delta\in \Der_{\KK^l}~|~ \delta(Q(\A)) \in \ideal{Q(\A)} S\}$. 

\begin{Definition} 
A central arrangement $\A$ in $\KK^l$ is said to be \textbf{free with exponents $(e_1,\dots,e_l)$} 
if and only if $D(\A)$ is a free $S$-module and there exists a basis $\delta_1,\dots,\delta_l \in D(\A)$ 
such that $\pdeg(\delta_i) = e_i$, or equivalently $D(\A)\cong\bigoplus_{i=1}^lS(-e_i)$.
\end{Definition}

Let $\delta_1,\dots,\delta_l \in D(\A)$. 
Then $\det(\delta_i(x_j))_{i,j}$ is divisible by $Q(\A)$.
The first characterization of freeness is due to Saito \cite{saito} and it uses  the determinant of the coefficient 
matrix of $\delta_1,\dots,\delta_l$ to check if the arrangement $\A$ is free or not. 

\begin{Theorem}[Saito's criterion]\label{theo:saitocrit}
Let $\A$ be a central arrangement in $\KK^l$ and $\delta_1, \dots, \delta_l \in D(\A)$. Then the following facts are equivalent
\begin{enumerate}
\item $D(\A)$ is free with basis $\delta_1, \dots, \delta_l$, i. e. $D(\A) = S\cdot\delta_1\oplus \cdots \oplus S \cdot\delta_l$.
\item $\det(\delta_i(x_j))_{i,j}=c Q(\A)$, where $c\in \KK\setminus\{0\}$.
\item $\delta_1, \dots, \delta_l$ are linearly independent over $S$ and $\sum_{i=1}^l\pdeg(\delta_i)=n$.
\end{enumerate}
\end{Theorem}

Given an arrangement $\A$ in $\KK^l$, the \textbf{Jacobian ideal} $J(\A)$ of $\A$
is the ideal of $S$ generated by $Q(\A)$ and all its partial derivatives.

The Jacobian ideal has a central role in the study of free arrangements.
In fact, we can also characterize freeness by looking at $J(\A)$ via the Terao's criterion.
Notice that Terao described this result for characteristic $0$, but the statement holds true for any characteristic as shown in \cite{palezzato2018free}.

\begin{Theorem}[\cite{terao1980arrangementsI}]\label{theo:freCMcod2} 
A central arrangement $\A$ in $\KK^l$ is free if and only if $S/J(\A)$ is $0$ or $(l-2)$-dimensional Cohen--Macaulay.
\end{Theorem}

In \cite{Gin-freearr}, the authors connected the study of generic initial ideals to the one of arrangements, obtaining a new characterization of freeness via the generic initial ideal of the Jacobian ideal.

\begin{Proposition}[\cite{Gin-freearr}]\label{prop:shapergin}
Let $\A =\{H_1, \dots, H_n\}$ be a central arrangement in $\KK^l$. 
Then $\rgin(J(\A))$ coincides with $S$ or 
its minimal generators include $x_1^{n-1}$, some positive power of
$x_2$, and no monomials only in $x_3,\dots, x_l$. 
\end{Proposition}

\begin{Example}\label{ex:NotFreeButLP}
Let $\A$ be the arrangement in $\RR^3$ with defining polynomial $Q(\A)=xyz(x+y+z)$.
In this case $\rgin(J(\A))=\ideal{x^3,x^2y,xy^2,y^4,y^3z}$.
\end{Example}

\begin{Theorem}[\cite{Gin-freearr}]\label{theo:firstequivfregin}
Let $\A =\{H_1, \dots, H_n\}$ be a central arrangement in $\KK^l$. 
Then $\A$ is free if and only if $\rgin(J(\A))$ coincides with $S$ or 
its minimal generators include $x_1^{n-1}$, some positive power of
$x_2$, and no monomials in $x_3,\dots, x_l$. 

More precisely, $\A$ is free if and only if
$\rgin(J(\A))$ coincides with $S$ or it is minimally generated by
$$x_1^{n-1},\; x_1^{n-2}x_2^{\lambda_1},\; \dots,\; x_2^{\lambda_{n-1}}$$ 
with $1\le\lambda_1<\lambda_2<\cdots<\lambda_{n-1}$ and $\lambda_{i{+}1}-\lambda_i= 1$ or $2$.
\end{Theorem}

\begin{Example}
Let $\A$ be the central arrangement in $\RR^3$ with defining polynomial $Q(\A)=xyz(x-y)(x-z)(y-z)$.
$\A$ is a free arrangement with exponents $(1,2,3)$.
In this case $\rgin(J(\A))=\ideal{x^5,x^4y,x^3y^2,x^2y^4,xy^5,y^7}$.
\end{Example}

The following Conjecture appeared in \cite{Gin-freearr}.

\begin{Conjecture}\label{conj:generatZ}
Let $\A$ be a central arrangement in $\KK^l$, 
and consider $d_0=\min\{d~|~x_2^{d}\in\rgin(J(\A))\}$.
If $\rgin(J(\A))$ has a minimal generator $t$ that involves the third variable of $S$, then $\deg(t)\ge d_0$.
\end{Conjecture}


\section{Hyperplane arrangements and Lefschetz properties}

In this section, we study the Jacobian algebra $S/J(\A)$ of an arrangement $\A$ from the point of view of the Lefschetz properties.

\begin{Proposition}\label{prop:l=2SLP}
Let $\A$ be a central arrangement in $\KK^2$. Then $S/J(\A)$ has the SLP.
\end{Proposition}
\begin{proof}
Since every central arrangement in $\KK^2$ is free, by Theorem~\ref{theo:firstequivfregin}, we have that 
$\rgin(J(\A)) = \ideal{x_1^{n-1},\; x_1^{n-2}x_2^{\lambda_1},\; \dots,\; x_2^{\lambda_{n-1}}}$.
This implies that $S/\rgin(J(\A))$ is Artinian.
From Proposition~3.15 in \cite{harima2013lefschetz}, every Artinian $\KK$-algebra $\KK[x_1,x_2]/I$ has the SLP, hence $S/\rgin(J(\A))$ has the SLP. We conclude by Proposition~\ref{prop:ILPginLP}.
\end{proof}

Not all arrangements have their Jacobian algebra that has the WLP.

\begin{Example}\label{ex:AnotWLPinR4}
Let $\A$ be the arrangement in $\RR^4$ with defining polynomial $Q(\A)=xyzw(x-y+z)(y+z-3w)(x+z+w)(x-5w)$.
In this case we have that $\HF(S/\rgin(J(\A)),9)=180$ and $\HF(S/\rgin(J(\A)),10)=207$.
This shows that the multiplication by $w$ from $(S/\rgin(J(\A)))_9$ to $(S/\rgin(J(\A)))_{10}$ is not surjective.
On the other hand, $x^2y^5z^2w$ is a minimal generator of $\rgin(J(\A))$ but $x^2y^5z^2\notin\rgin(J(\A))$, 
and hence the multiplication by $w$ from degree $9$ to degree $10$ is not even injective.
This shows that $w$ is not a Lefschetz element for $S/\rgin(J(\A))$, and hence, by Lemma~\ref{lemma:SSWLPx_l},
$S/\rgin(J(\A))$ does not have the WLP.
By Proposition~\ref{prop:ILPginLP}, also $S/J(\A)$ does not have the WLP.
\end{Example}

The freeness of an arrangement $\A$ forces their Jacobian algebra $S/J(\A)$ to have the SLP.

\begin{Theorem}\label{thm:FreeSLP}
Let $\A$ be a free arrangement in $\KK^l$. Then $S/J(\A)$ has the SLP.
\end{Theorem}
\begin{proof}
If $l=2$ we can directly conclude by Proposition~\ref{prop:l=2SLP}.
Assume $l\geq3$, by Theorem~\ref{theo:firstequivfregin},
$\rgin(J(\A)) = \ideal{x_1^{n-1},\; x_1^{n-2}x_2^{\lambda_1},\; \dots,\; x_2^{\lambda_{n-1}}}$.
This implies that the multiplication by $x_l^k$ from $(S/\rgin(J(\A)))_i$ to $(S/\rgin(J(\A)))_{i+k}$ is injective for all $i \geq 0$ and $k \geq 1$, and hence $S/\rgin(J(\A))$ has the SLP.
We conclude by Proposition~\ref{prop:ILPginLP}.
\end{proof}

Notice that Theorem~\ref{thm:FreeSLP} is not an equivalence.

\begin{Example}
Let $\A$ be the arrangement in $\RR^3$ of Example~\ref{ex:NotFreeButLP}.
$\A$ is non-free and a direct computation shows that $z$ is a strong Lefschetz element for $S/\rgin(J(\A))$.
Hence, $S/\rgin(J(\A))$ has the SLP.
By Proposition~\ref{prop:ILPginLP}, also $S/J(\A)$ has the SLP.
\end{Example} 

The freeness of an arrangement $\A$ forces the Hilbert function of $S/J(\A)$ to be an increasing function.

\begin{Theorem}\label{thm:FreeIncreasingHF}
Let $\A$ be a free arrangement in $\KK^l$. Then the Hilbert function of $S/J(\A)$ is an increasing function.
\end{Theorem}
\begin{proof}
As seen in the proof of Theorem~\ref{thm:FreeSLP}, the freeness of $\A$ implies that the multiplication by $x_l$ from 
$(S/\rgin(J(\A)))_i$ to $(S/\rgin(J(\A)))_{i+1}$ is injective for all $i \geq 0$. This implies that 
$\dim_\KK((S/\rgin(J(\A)))_i) \leq \dim_\KK((S/\rgin(J(\A)))_{i+1})$
for all $i \geq 0$, and hence the Hilbert function of $S/J(\A)$ is an increasing function, by Remark~\ref{rem:rginsameHFasideal}.
\end{proof}

Notice that Theorem~\ref{thm:FreeIncreasingHF} is not an equivalence.

\begin{Example}
Let $\A$ be the arrangement in $\RR^4$ with defining polynomial $Q(\A) = xw(x+y+w)(x-y)(x+z)$. The arrangement $\A$ is not free, however the Hilbert function of the $S/J(\A)$ is a strictly increasing function.
In fact, we have that $\HF(S/J(\A),d) = (1,4,10,20,31,41)$ for $d=0,\dots,5$,
and $\HF(S/J(\A),d+1) = \HF(S/J(\A),d)+10$, for all $d \geq 5$.
\end{Example}

In general the Hilbert function of $S/J(\A)$ is not a monotonic function.

\begin{Example}
Let $\A$ be the arrangement in $\RR^3$ of Example~\ref{ex:NotFreeButLP}.
Then $\HF(S/J(\A),2)=6$, $\HF(S/J(\A),3)=7$ but $\HF(S/J(\A),4)=6$.
\end{Example}

If Conjecture~\ref{conj:generatZ} holds, this would give information on the Jacobian algebra of arrangements in $\KK^3$. 

\begin{Proposition}
Let $\A$ be a central arrangement in $\KK^3$. 
If Conjecture~\ref{conj:generatZ} holds, then $S/J(\A)$ has the WLP.
\end{Proposition}
\begin{proof}
If $\A$ is free, then $S/J(\A)$ has always the WLP by Theorem~\ref{thm:FreeSLP}.
Let $\A$ be a non-free arrangement in $\KK^3$.
From Proposition~\ref{prop:shapergin}, $\rgin(J(\A))$ is minimally generated by
$\{x_1^{n-1},\; x_1^{n-2}x_2^{\lambda_1},\; \dots,\; x_2^{\lambda_{n-1}}\}$ and some other monomials involving $x_3$.
Let $\gamma$ be the smallest degree of a minimal generator of $\rgin(J(\A))$ divisible by $x_3$.  
If Conjecture~\ref{conj:generatZ} holds, then $\gamma \geq \lambda_{n-1}$.


For all $i=0,\dots,\gamma-2$ the multiplication by $x_3$ from $(S/\rgin(J(\A)))_i$ to $(S/\rgin(J(\A)))_{i+1}$ is clearly injective. 
Consider $i\geq\gamma-1$. Since $\gamma \geq \lambda_{n-1}$, $(S/\rgin(J(\A)))_{i+1}$ is generated by monomials that are all divisible by $x_3$. This shows that the multiplication by $x_3$ from $(S/\rgin(J(\A)))_i$ to $(S/\rgin(J(\A)))_{i+1}$ is surjective. This implies that $S/\rgin(J(\A))$ has the WLP.
We conclude by Proposition~\ref{prop:ILPginLP}.
\end{proof}

In general, if we consider strongly stable ideals in three variables that look like the ones described in Proposition~\ref{prop:shapergin}, their associated quotient algebra might not have the WLP.

\begin{Example}
Let $J=\ideal{x^3,x^2y,x^2z,xy^3,y^5}$ be a strongly stable ideal  in $S=\RR[x,y,z]$.
Then $S/J$ does not have the WLP. In fact, $\HF(S/J,2)=6$ and $\HF(S/J,3)=7$. 
However, the multiplication by $z$ from $(S/J)_2$ to $(S/J)_3$ is not injective, and hence by Lemma~\ref{lemma:SSWLPx_l}, $S/J$ does not have the WLP.
\end{Example}

\begin{Proposition}
Let $\A$ be a free arrangement in $\KK^l$. Then $\rgin(J(\A))$ is an almost revlex ideal.
\end{Proposition}
\begin{proof} It is a direct consequence of Definition~\ref{def:ARL} and Theorem~\ref{theo:firstequivfregin}.
\end{proof}

In general, a central arrangement $\A$ might have $\rgin(J(\A))$ that is not an almost revlex ideal but $S/J(\A)$ has the SLP. Moreover, it might also be that $S/J(\A)$ does not have the SLP.

\begin{Example} Let $\A$ be the arrangement in $\mathbb{R}^5$ with defining polynomial $Q(\A)=xyw(x-y+z+w)(x+v)(x+y+w+v)(z-w-v)\in\mathbb{R}[x,y,z,w,v]$. Then $y^6zw$ is a minimal generator of $\rgin(J(\A))$. However, $z^{8}>_{\DegRevLex}y^6zw$ and $z^{8}\notin\rgin(J(\A))$. This shows that $\rgin(J(\A))$ is not an almost revlex ideal. However, $\rgin(J(\A))$ has no minimal generator divisible by $v$ and hence $S/\rgin(J(\A))$ has the SLP. This implies, by Proposition~\ref{prop:ILPginLP}, that $S/J(\A)$ has the SLP.
\end{Example}
 
\begin{Example} Let $\A$ be the arrangement in $\mathbb{R}^4$ of Example~\ref{ex:AnotWLPinR4}. 
Then $y^8w^3$ is a minimal generator of $\rgin(J(\A))$. However, $z^{11}>_{\DegRevLex}y^8w^3$ and $z^{11}\notin\rgin(J(\A))$. This shows that $\rgin(J(\A))$ is not an almost revlex ideal.
In addition, as seen in  Example~\ref{ex:AnotWLPinR4}, $S/\rgin(J(\A))$ does not have the WLP and hence the SLP. However, the fact that $S/\rgin(J(\A))$ does not have the SLP can also be seen directly from the fact that since $y^8\notin \rgin(J(\A))$ and $y^8w^3\in\rgin(J(\A))$, the multiplication map by $w^3$ does not have maximal rank from degree $8$ to degree $11$.
\end{Example}

Notice that all the examples of central arrangements $\A$ in $\KK^3$, that we have considered so far, have $\rgin(J(\A))$ that is an almost revlex ideal. By Corollary~\ref{corol:ALMhasSLP} and Proposition~\ref{prop:ILPginLP}, all considered arrangements in $\KK^3$ have the property that $S/J(\A)$ has the SLP.

\bigskip
\paragraph{\textbf{Acknowledgements}} The authors would like to thank Y. Numata for many helpful discussions. During the preparation of this article the second author was supported by JSPS Grant-in-Aid for Early-Career Scientists (19K14493).





\end{document}